\documentclass[11pt,a4paper]{article}

\usepackage[hmargin={25mm,25mm},vmargin={25mm,25mm}]{geometry}
\usepackage[utf8]{inputenc}
\usepackage{latexsym}
\usepackage{amssymb}
\usepackage{amsfonts}
\usepackage{mathtools}
\usepackage{commath}
\usepackage{bbm}
\usepackage{csquotes}
\usepackage{theoremref}
\usepackage{gensymb}
\usepackage{tikz}
\usepackage{authblk}

\usepackage{amsthm, mathtools, graphicx, paralist, cite, color, url, hyperref}
\usepackage{comment}
\usepackage{hyperref}

\usepackage{amsmath, nccmath}
\usepackage{indentfirst}
\usepackage{appendix}
\usepackage{listings}
\usepackage{hyphenat}
\usepackage{physics}
\usepackage{calrsfs}
\usepackage{enumerate}
\usepackage{tikz-cd}
\usepackage{mathrsfs}

\DeclareMathAlphabet{\mathcal}{OMS}{cmsy}{m}{n}
\SetMathAlphabet{\mathcal}{bold}{OMS}{cmsy}{b}{n}

\theoremstyle{plain}
\newtheorem{theorem}{Theorem}[section]
\newtheorem{lemma}[theorem]{Lemma}
\newtheorem{prop}[theorem]{Proposition}
\newtheorem{cor}[theorem]{Corollary}

\newtheorem{obs}[theorem]{Observation}

\theoremstyle{definition}

\newtheorem{remark}{Remark}

\begin{document}

\title{A two-player voting game in Euclidean space}
 
\author{Stelios Stylianou\footnote{School of Mathematics, University of Bristol. Supported by an EPSRC Doctoral Training Studentship.}}
\date{November 2025}
\maketitle

\begin{sloppypar}

\begin{abstract}
Given a finite set $S$ of points in $\mathbb{R}^d$, which we regard as the locations of voters on a $d$-dimensional political `spectrum', two candidates (Alice and Bob) select one point in $\mathbb{R}^d$ each, in an attempt to get as many votes as possible. Alice goes first and Bob goes second, and then each voter simply votes for the candidate closer to them in terms of Euclidean distance. If a voter's distance from the two candidates is the same, they vote for nobody. We give a geometric characterization of the sets $S$ for which each candidate wins, assuming that Alice wins if they get an equal number of votes. We also show that, if not all the voters lie on a single line, then, whenever Alice has a winning strategy, there is a unique winning point for her. We also provide an algorithm which decides whether Alice has a winning point, and determines the location of that point, both in finite (in fact polynomial) time.
\end{abstract}

\section{Introduction}

Let $d$ be a positive integer, and let $S$ be a finite set of distinct points in $\mathbb{R}^d$, which we regard as the locations of voters. We will sometimes use the term $\emph{voters}$ to refer to elements of $S$. Let $\rho(x,y) \coloneqq \norm{x-y}_2 = ({\sum_{i=1}^{d}(x_i-y_i)^2})^{1/2}$ denote the Euclidean distance between $x,y \in \mathbb{R}^d$. Two players, Alice and Bob, who will also be referred to as the $\emph{candidates}$, play the following game: Alice chooses some point $a \in \mathbb{R}^d$. Bob knows $a$, and chooses some point $b \in \mathbb{R}^d$. The $\emph{scores}$ of Alice and Bob are given by $V_A \coloneqq \abs{\{ x \in S: \rho(a,x) < \rho(b,x) \}}$ and $V_B \coloneqq \abs{\{ x \in S: \rho(b,x) < \rho(a,x) \}}$, respectively. In other words, each voter will vote for the candidate closer to them in terms of Euclidean distance, and any voter that is equidistant from the two candidates will not vote. We say that a candidate $\emph{claims}$ the point $x \in S$ if $x$ votes for them.

We assume that both Alice and Bob execute perfect strategies, so `Alice/Bob wins' will have the same meaning as `Alice/Bob has a winning strategy'. Note that Bob can ensure that $V_B=V_A$ by choosing the same point as Alice did (both scores will be equal to zero in this case). We will therefore say that Alice wins if and only if $V_A \geq V_B$, which will also allow us to assume that Bob always chooses a different point from Alice. 

Note that the problem is trivial when $d=1$: let $S=\{ x_1, \dots, x_n\}$, where $x_1< \dots <x_n$. Alice can always win by picking the median value of $S$. More precisely, if $n$ is odd, then Alice wins if and only if she picks $a=x_{(n+1)/2}$, and, if $n$ is even, then Alice wins if and only if she picks some $a \in [x_{n/2}, x_{n/2+1}]$. This (one-dimensional) result is referred to as the $\emph{median voter theorem}$. A version of this result was first observed by Hotelling \cite{Hotelling} in 1929, and it was more formally stated by Black \cite{Black} in 1948.

Part of the motivation behind solving this problem in higher dimensions arises from the fact that, in political science, the positions of voters and candidates are often represented by points in a two-dimensional space (or higher-dimensional space), with each dimension corresponding to a different issue. For example, the horizontal and vertical axes of a two-dimensional model often represent opinions on economic policy and social policy, respectively, on a scale from `left-wing' to `right-wing', or `liberal' to `conservative'. A third axis could, for example, correspond to opinions on foreign policy. See \cite{enelow1984spatial} and \cite{Unified} for a discussion of prior work in social choice theory, on spatial voting models. It is also worth mentioning that researchers in the social sciences sometimes attempt to identify the position of potential voters in a two-dimensional space (e.g.\ with a view to comparing them with the approximate positions of political parties), by using a questionnaire such as \cite{electoralcalculus2025}. In the economics literature, a similar two-player game is often used to model the behaviour of two vendors competing for market-share, though there it is more usual to model the market/customers as a continuum of points rather than a finite set of points (e.g., as a unit line segment, a unit circle, a unit square or a unit disc), and to allow the two vendors to place more than one vending outlet in the space, either alternately or with one vendor placing all of her outlets, followed by the other vendor placing all of his. Similar games with more than two vendors have also been studied. We refer the reader to \cite{linial} and the references therein, for several interesting results in this setting.

In this paper, we will provide a geometric characterization of the finite sets $S$ for which each candidate wins, for all $d \geq 2$. Our main result will distinguish between the case where $S$ consists of an odd number of voters, and the case where it consists of an even number of voters, but we will first look at a simple necessary and sufficient condition for Alice to win that works for both odd-sized and even-sized sets.

\begin{theorem}
\label{general_easy}
    Alice wins if and only if there exists some point $x \in \mathbb{R}^d$ such that any affine hyperplane through $x$ contains at most half the points of $S$ on either (strict) side. Moreover, provided that not all voters lie on a single line, Alice cannot have two distinct winning points.
\end{theorem}

\begin{proof}
    Suppose there exists some point $x$ that satisfies the condition of the theorem, and Alice chooses $a=x$. Then, whatever choice of $b$ Bob makes, we can think of $a$ as being the projection of $b$ onto some affine hyperplane $H$ through $a$. Since at most half the points of $S$ lie on the same side of $H$ as $b$, and Alice claims all the remaining points, she wins.

    Conversely, if no such point $x$ exists, then, whatever the choice of $a$ is, Bob can find some affine hyperplane $H$ through $a$ that contains more than half the points of $S$ on one side. We can assume, without loss of generality, that $a$ is the origin, $H= \{ x \in \mathbb{R}^d:x_1=0 \}$, and $H^+ = \{ x \in \mathbb{R}^d : x_1>0\}$ contains more than half the points of $S$. Then Bob wins by choosing $b=(\beta,0,\dots,0)$, for some sufficiently small $\beta>0$.

    Now suppose that not all voters lie on a single line. Assume that $a_1$ and $a_2$ are two distinct winning points for Alice, and let $l$ be the line through these two points. Fix some $y \in S$ that does not belong to $l$. We can find affine hyperplanes $H_1$ and $H_2$, containing $a_1$ and $a_2$, respectively, such that the normal vector to $H_1$ is the same as the normal vector to $H_2$, and $y$ lies strictly between $H_1$ and $H_2$. Without loss of generality, we can take $H_1 = \{ x \in \mathbb{R}^d:x_1=0 \}$, and $H_2 = \{ x \in \mathbb{R}^d:x_1= \alpha \}$, for some $\alpha>0$. Let $k_1=\abs{\{x \in S: x_1 \leq 0 \}}$, $k_2=\abs{\{x \in S: x_1 \geq \alpha \}}$, and $k_3=\abs{\{x \in S: 0<x_1<\alpha \}}$. Then $k_1+k_2+k_3=\abs{S}$. Since $a_1$ is winning for Alice, we know that $k_2+k_3 \leq \abs{S}/2$, and, since $a_2$ is also winning for Alice, we know that $k_1+k_3 \leq \abs{S}/2$. Adding these two inequalities gives $k_1+k_2+2k_3 \leq \abs{S}$, so $k_3 = 0$, a contradiction since $y$ satisfies $0<x_1<\alpha$.
\end{proof}

\begin{remark}
    This implies that, when $\abs{S}$ is odd, Alice can only hope to win by choosing $a$ to be some point of $S$. When $\abs{S}$ is even, Bob can find an affine hyperplane containing $a$ and no other point of $S$. Then one of the two open half-spaces defined by $H$ will contain at least half the points of $S$, and Bob will be able to claim all those points, by following the strategy described in the second paragraph of the proof. Therefore, even if Bob is not allowed to choose $b=a$, he can always claim at least half of the votes when there is an even number of voters, which further justifies declaring Alice the winner if $V_A=V_B$. 
\end{remark}

Henceforth, we will always work with $d \geq 2$. Let $l$ be a line containing the point $x$, and let $l^+$ and $l^-$ be the two open half-lines defined by $x$ on $l$. We say that $l$ is:
\begin{itemize}
    \item \emph{Balanced (about $x$)}, if $\abs{\abs{l^+ \cap S} - \abs{l^- \cap S}} \leq 1$.
    \item \emph{Even-balanced (about $x$)}, if $l$ is balanced and $\abs{(l \setminus x) \cap S}$ is even.
    \item \emph{Odd-balanced (about $x$)}, if $l$ is balanced and $\abs{(l \setminus x) \cap S}$ is odd.
\end{itemize}

In other words, a line is even-balanced about $x$ if there is an equal number of voters on either side of $x$, whereas a line is odd-balanced about $x$ if the number of voters on one side differs to the number of voters on the other side by one. Note that, if $l\setminus x$ contains no points of $S$, then $l$ is trivially even-balanced about $x$. If $l$ is even-balanced, and $(l \setminus x) \cap S$ is non-empty, we will say that $l$ is a \emph{non-trivial} even-balanced line.

The set of vectors $\{e_1,\dots,e_d\}$ is used to denote the standard basis in $\mathbb{R}^d$. For $x,y \in \mathbb{R}^d$, we write $\langle x,y \rangle \coloneqq \sum_{i=1}^{d}x_iy_i$ for the standard inner product on $\mathbb{R}^d$. An \emph{affine hyperplane} in $\mathbb{R}^d$ is a set $H$ of the form $\{ x \in \mathbb{R}^d: \langle x,c \rangle = \lambda\}$, for some non-zero $c \in \mathbb{R}^d$, and some $\lambda \in \mathbb{R}$. For simplicity, we will usually use the term \emph{hyperplane} instead. For a hyperplane $H=\{ x \in \mathbb{R}^d: \langle x,c \rangle = \lambda\}$, let $H^+ \coloneqq \{ x \in \mathbb{R}^d: \langle x,c \rangle > \lambda \}$, and $H^- \coloneqq \{ x \in \mathbb{R}^d: \langle x,c \rangle < \lambda \}$. Note that this does not uniquely define $H^+$ and $H^-$: if we replace $c$ and $\lambda$ by $-c$ and $-\lambda$, respectively, then $H$ is the same, but the roles of $H^+$ and $H^-$ are reversed. However, this will not have any effect on our proofs.

When $\abs{S}$ is odd, we say that $H$ is:
\begin{itemize}
    \item \emph{Good}, if $\abs{H^+ \cap S} = \abs{H^- \cap S}$.
    \item \emph{Perfect}, if $H$ is good and $\abs{H \cap S} = 1$.
\end{itemize}
Note that a good hyperplane contains at least one point of $S$, since $\abs{S}$ is odd. Moreover, it is easy to see that at least one perfect hyperplane exists. Indeed, choose some non-zero $u \in \mathbb{R}^d$ such that the inner products $\langle u,y\rangle_{y \in S}$ are all distinct; then taking the median value of these inner products ($\lambda$, say) yields a perfect hyperplane $\{x \in \mathbb{R}^d:\ \langle x,u\rangle = \lambda\}$.

Our main result for odd-sized sets gives three necessary and sufficient conditions in order for Alice to win. The first two conditions are quite similar, as they refer to good hyperplanes and perfect hyperplanes through some point of $S$, respectively, whereas the third condition is the simplest one, as it only refers to lines through some point of $S$.

\begin{theorem}
\label{thm_odd_general}
If $\abs{S}$ is odd, then Alice wins if and only if there exists a point $x \in S$ that satisfies any one of the following equivalent conditions:

\begin{enumerate}[(i)]
    \item Every hyperplane through $x$ is good. \label{condition1_odd}
    \item Every perfect hyperplane contains $x$. \label{condition2_odd}
    \item Every line through $x$ is even-balanced. \label{condition3_odd}
\end{enumerate}
In this case, $a=x$ is a winning point for Alice.
\end{theorem}

When $\abs{S}$ is even, we say that $H$ is \emph{perfect}, if $\abs{H^+ \cap S} < \abs{S}/2$ and $\abs{H^- \cap S} < \abs{S}/2$. Our main result for even-sized sets gives two necessary and sufficient conditions in order for Alice to win. The first one is exactly the same as the condition regarding perfect hyperplanes in the odd-sized case, whereas the second one again asks for every line through some point to be balanced, with an additional requirement regarding the odd-balanced lines.

\begin{theorem}
\label{thm_even_general}
If $\abs{S}$ is even, then Alice wins if and only if there exists a point $x \in \mathbb{R}^d$ that satisfies any one of the following equivalent conditions:
\begin{enumerate}[(i)]
    \item Every perfect hyperplane contains $x$. \label{condition1_even}
    \item Every line through $x$ is balanced, and there is a plane $P$ that contains all the odd-balanced lines about $x$. If there are no odd-balanced lines about $x$, then $x \notin S$. If there exists at least one odd-balanced line about $x$, then $x \in S$. Moreover, if $L^*=\{l_1,\dots,l_k\}$ is the set of odd-balanced lines, labeled in a clockwise order on $P$, where the half-lines appear in the order $l_1^+,\dots,l_k^+,l_1^-,\dots,l_k^-$, then $k$ is odd, and either $\abs{l_i^+ \cap S} = \abs{l_i^- \cap S}+1$ precisely when $i$ is odd, or $\abs{l_i^+ \cap S} = \abs{l_i^- \cap S}+1$ precisely when $i$ is even.
    \label{condition2_even}
\end{enumerate}
In this case, $a=x$ is a winning point for Alice.
\end{theorem}

Note that, even though condition (\ref{condition2_even}) looks more complicated, it is in fact easier to visualize, and it will also be useful when we consider the algorithmic problem of efficiently locating winning points.

For both main results, we will begin by proving the corresponding two-dimensional versions. Then, for the odd case, we will mimic the two-dimensional proof in order to get the result in higher dimensions. For the even case, we will reduce the higher dimensional case into the two-dimensional one, by projecting $S$ onto a plane.

It is easy to see that, for both odd-sized and even-sized sets, the perfect hyperplanes having a common point is a necessary condition, since in a different case Bob could find some perfect hyperplane $H$ that does not contain $a$, and win by choosing the projection of $a$ onto $H$.

For odd-sized sets, it is also easy to see that the other two conditions are sufficient. For the first condition, we just need to look at the good hyperplane which is perpendicular to the line through $a$ and $b$. For the third condition, we observe that, for any line $l$ through $a$, Alice can always claim all the points of $S$ on one of the two half-lines defined by $a$ on $l$. In order to show that the perfect hyperplanes being concurrent is sufficient, we will rely on the fact that we can always find `a lot of' perfect hyperplanes.

For even-sized sets in two dimensions, we will first show that there exists a perfect line through any point of $S$. We will then show that, if the perfect lines are concurrent, but the condition on balanced lines is not satisfied, then we can in fact find some new perfect line which is not concurrent with the other ones. If the balanced lines condition is satisfied, then Alice will again be able to claim at least half of the votes from every even-balanced line. Asking for the odd-balanced lines to `alternate' in the way described above also guarantees that Alice claims enough votes on those lines in order to win.

Once we have proven the two main results, we will demonstrate that, unless $d=1$, or $d=2$ and there is an even number of voters, the set of Alice wins is `very small'. More precisely, we will show that any set of at least three points in general position is winning for Bob. This implies that, for fixed $n \geq 3$, the set of ordered $n$-tuples that correspond to an Alice win is nowhere dense, as a subset of the metric space $M_{n,d} \coloneqq \{ (x_1, \dots, x_n) \in (\mathbb{R}^d)^n: \text{all the} \ x_i \ \text{are distinct}\}$, endowed with the Euclidean metric.

We will then provide algorithms that use the condition on balanced lines in order to determine whether Alice wins, and also locate her unique winning point (provided that the voters do not lie on a single line, in which case it is easy to find the set of winning points for Alice). Our algorithms run in polynomial time.

At the end of the paper, we will prove that, whatever the structure of $S$ is, Alice can guarantee that she wins at least a proportion $\frac{1}{d+1}$ of the votes. The (short) proof of this result is due to the author, David Ellis and Robert Johnson.
 
\section{The result in two dimensions}

As we are now working in two dimensions, a hyperplane will simply be a line. In this section, $l^+$ and $l^-$ will be used to denote the two open half-lines defined by some point $x$ on the line $l$, rather than the two open half-planes into which $l$ divides the plane. We begin with a simple observation regarding the points claimed by Alice:

\begin{obs}
\label{obs_weak}
    Let $l$ be the perpendicular to the line $ab$ at $a$. Then Alice claims all the points in the closed half-plane defined by $l$ that does not contain $b$.
\end{obs}

In fact, what we have just described is a subset of the set that contains all the points claimed by Alice. We can improve our previous observation in order to define this bigger set:

\begin{obs}
\label{obs_perp_bis}
    Let $l$ be the perpendicular bisector of the line $ab$. Then Alice claims precisely those points that belong to the open half-plane defined by $l$ that contains $a$.
\end{obs}

This shows that, for fixed $a$, Bob's score can never decrease if he decides to move from $b$ to $b'$, for some $b'$ strictly between $a$ and $b$ on the line $ab$. If Bob chooses a point sufficiently close to $a$, then all the points claimed by Alice belong to the set given by Observation \ref{obs_weak}.

\subsection{Odd number of voters}

As we have already mentioned, the notion of perfect lines is important since Alice has to position herself on every such line:
\begin{obs}
\label{obs_odd_2d}
If the line $l$ is perfect (or just good), and $a \notin l$, then Bob wins by choosing $b$ to be the projection of $a$ onto $l$.
\end{obs}

Moreover, a lot of perfect lines exist:
\begin{lemma}
\label{lem_2d_odd_perfect}
When $d=2$ and $\abs{S}$ is odd, a perfect line exists in all but finitely many directions.
\end{lemma}

\begin{proof}
Let $m \in \mathbb{R}$, and consider lines of the form $y=mx+c$, as $c$ varies over $\mathbb{R}$. If each such line contains at most one point of $S$, then clearly some perfect line with gradient $m$ exists. (More precisely, let $c_1, \ldots, c_n$ be the choices of $c$ for which $y=mx+c$ contains exactly one point of $S$, and let $c^*$ be the median value of the $c_i$. Then the line given by $y=mx+c^*$ is perfect.) But the number of lines that pass through at least two points of $S$ is at most $\binom{\abs{S}}{2}$, which means that a perfect line exists in all but at most $\binom{\abs{S}}{2}$ directions.
\end{proof}

\begin{theorem}

When $d=2$ and $\abs{S}$ is odd, Alice wins if and only if there exists a point $x \in S$ that satisfies any one of the following equivalent conditions:

\begin{enumerate}[(i)]
    \item Every line through $x$ is good. \label{condition1_2d_odd}
    \item Every perfect line contains $x$. \label{condition2_2d_odd}
    \item Every line through $x$ is even-balanced. \label{condition3_2d_odd}
\end{enumerate}
In this case, $a=x$ is a winning point for Alice.
\end{theorem}

\begin{proof}

Suppose there exists $x$ satisfying (\ref{condition1_2d_odd}). Alice chooses $a=x$. Then, whatever the choice of $b$ is, we can think of $a$ as being the projection of $b$ onto some good line, showing that Alice wins.

Now suppose that there exists $x$ satisfying (\ref{condition3_2d_odd}). Alice chooses $a=x$. Whatever the choice of $b$ is, Alice claims $x$, and, for any line $l$ through $x$, either all points of $S$ on $l^+$, or all points of $S$ on $l^-$. Therefore, Alice gets more than half of the votes, and she wins.

If Alice wins, then Observation \ref{obs_odd_2d} implies that (\ref{condition2_2d_odd}) necessarily holds.

Suppose now that $x$ satisfies (\ref{condition2_2d_odd}). Then, if Alice is to win, she must choose $x$. Since a perfect line exists in all but finitely many directions, all but finitely many lines through $x$ are perfect. Let $l_0$ be any line through $x$. Let $k$ and $k'$ be the number of points of $S$ in the two open half-planes defined by $l_0$, and let $p$ and $p'$ be the number of points of $S$ on $l_0^+$ and $l_0^-$, respectively. We can rotate $l_0$ clockwise about $x$ by some sufficiently small angle, so that our line does not `meet' any points of $S$ while rotating, and get a perfect line $l_1$. Similarly, we can get a perfect line $l_2$ by rotating anticlockwise. Since both $l_1$ and $l_2$ are perfect, we deduce that $k+p=k'+p'$, and $k+p'=k'+p$, which gives $k=k'$, and $p=p'$, proving both (\ref{condition1_2d_odd}) and (\ref{condition3_2d_odd}).
\end{proof}

\begin{figure}
    \centering
    
    \begin{tikzpicture}
    \draw[black, thick] (-3,-2) -- (3,2);
    \filldraw[black] (-1,-2/3) circle (1.5pt);
    \filldraw[black] (-1.5,-1) circle (1.5pt);
    \filldraw[black] (1.2,0.8) circle (1.5pt);
    \filldraw[black] (2.7,1.8) circle (1.5pt);
    
    \draw[black, thick] (1.5,-2) -- (-1.5,2);
    \filldraw[black] (1,-4/3) circle (1.5pt);
    \filldraw[black] (0.6,-0.8) circle (1.5pt);
    \filldraw[black] (-0.3,0.4) circle (1.5pt);
    \filldraw[black] (-1.2,1.6) circle (1.5pt);
    
    \draw[black, thick] (-1,-2) -- (1,2);
    \filldraw[black] (-0.9,-1.8) circle (1.5pt);
    \filldraw[black] (-0.8,-1.6) circle (1.5pt);
    \filldraw[black] (-0.6,-1.2) circle (1.5pt);
    \filldraw[black] (0.2,0.4) circle (1.5pt);
    \filldraw[black] (0.45,0.9) circle (1.5pt);
    \filldraw[black] (0.9,1.8) circle (1.5pt);
    
    \draw[black, thick] (2.5,-2) -- (-2.5,2);
    \filldraw[black] (2,-1.6) circle (1.5pt);
    \filldraw[black] (-1,0.8) circle (1.5pt);
    
    \draw[black, thick] (-3,-1) -- (3,1);
    \filldraw[black] (-2,-2/3) circle (1.5pt);
    \filldraw[black] (1.5,0.5) circle (1.5pt);
    
    \filldraw[black] (0,0) circle (1.5pt) node[anchor=west] {$a$};
\end{tikzpicture}

    \caption{A winning set for Alice with $\abs{S}$ odd and $a \in S$, or with $\abs{S}$ even and $a \notin S$.}
    \label{fig:odd_2d}
\end{figure}
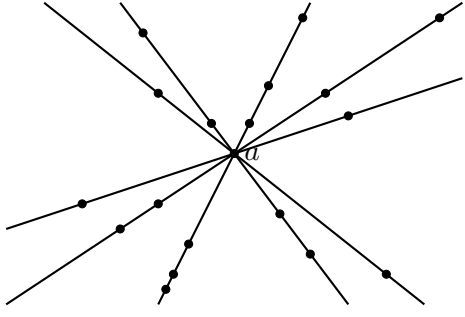

\subsection{Even number of voters}
Now suppose $S$ consists of an even number of points. Note that (the analogue of) Observation \ref{obs_odd_2d} holds true:

\begin{obs}
\label{obs_even_2d}
If the line $l$ is perfect, and $a \notin l$, then Bob wins by choosing $b$ to be the projection of $a$ onto $l$.
\end{obs}

The existence of lots of perfect lines might not be as obvious as in the case where $\abs{S}$ is odd, but we can indeed prove a similar result:

\begin{lemma}
\label{lem_2d_even_perfect}
Let $S$ be a set of $2n$ points in $\mathbb{R}^2$. Then there exists a perfect line through any point of $S$.
\end{lemma}

\begin{proof}
Fix an arbitrary point $x \in S$. Let $l_1, \dots, l_L$ be the collection of all lines through $x$ that contain at least one other point of $S$, labeled in a clockwise order. Note that $L \leq 2n-1$. Fix a direction on $l_1$ (called `left to right') and obtain the corresponding direction on $l_i$ by rotating $l_1$ clockwise until it meets $l_i$. Let $l_i^+$ be the open half-line to the right of $x$, and $l_i^-$ the open half-line to the left of $x$. Let $u_i \coloneqq \abs{l_i^+ \cap S}$, and $v_i \coloneqq \abs{l_i^- \cap S}$. Then $T \coloneqq \sum_{i=1}^L (u_i+v_i) = 2n-1$. Let $T_i \coloneqq u_{i+1}+\dots+u_L+v_1+\dots+v_{i-1}$, and $T'_i \coloneqq v_{i+1}+\dots+v_L+u_1+\dots+u_{i-1}$. In other words, $T_i$ and $T'_i$ are the number of points of $S$ contained in each of the two (open) half-planes defined by $l_i$. Finding a perfect line through $x$ is then equivalent to showing that there exists $1 \leq i \leq L$ such that $T_i < n$ and $T'_i < n$.

Note that, for each $i$, either $T_i < n$ or $T'_i < n$ must hold. Now assume that exactly one of these holds for any fixed $i$. Then suppose, without loss of generality, that $T_1 \geq n$. If $T_L \geq n$, then $2n \leq T_1+T_L = (u_2+\dots+u_L)+(v_1+\dots+v_{L-1}) \leq T = 2n-1$, a contradiction. Therefore, $T'_L \geq n$. If $T_{L-1} \geq n$, then $2n \leq T'_L+T_{L-1} = (u_1+\dots+u_{L-1})+(u_L+v_1+\dots+v_{L-2}) \leq T = 2n-1$, a contradiction. Therefore, $T'_{L-1} \geq n$. Similarly, we get $T'_i\geq n$ for $2 \leq i \leq L-2$. But then $T'_2 \geq n$ would imply $T'_1 \geq n$, which contradicts $T_1 \geq n$.
\end{proof}

We will now explain why a set that either looks like the one from Figure \ref{fig:odd_2d} (with $a \notin S$), or the one from Figure \ref{fig:even_2d} (with $a \in S$), is winning for Alice:

\begin{prop}
\label{prop_even}
Let $S$ be a set of $2n$ points in $\mathbb{R}^2$. If either of the following two conditions holds, then Alice wins:

\begin{itemize}
    
    \item All the perfect lines are concurrent at some $x \notin S$, and every perfect line is even-balanced.
    
    \item All the perfect lines are concurrent at some $x \in S$, and every perfect line is balanced. Additionally, if $L^*=\{l_1,\dots,l_k\}$ is the set of odd-balanced perfect lines (in a clockwise order), where the half-lines appear in the order $l_1^+,\dots,l_k^+,l_1^-,\dots,l_k^-$, then $k$ is odd, and either $\abs{l_i^+ \cap S} = \abs{l_i^- \cap S}+1$ precisely when $i$ is odd, or $\abs{l_i^+ \cap S} = \abs{l_i^- \cap S}+1$ precisely when $i$ is even.
\end{itemize}
In either case, $a=x$ is a winning point for Alice.
\end{prop}

\begin{proof}
If the first condition holds, and Alice chooses $a=x$, then she claims at least half the points of $S$ on every even-balanced line, and therefore she wins.

If the second condition holds, and Alice chooses $a=x$, then she claims $x$, and also at least half the points of $S$ on every even-balanced line. Now consider the lines in $L^*$, and suppose they contain a total of $2N+1$ points of $S$ (excluding $x$). Using Observation \ref{obs_weak}, we see that, whatever the choice of $b$ is, Alice claims all the points on $k$ consecutive (in a rotational sense) half-lines from the set $\{l_1^+, \dots, l_k^+, l_1^-, \dots, l_k^-\}$. But our condition implies that these $k$ half-lines must contain at least $N$ points of $S$. Therefore, Alice claims at least $1+\frac{2n-(2N+2)}{2}+N = n$ points, and therefore she wins.
\end{proof}

\begin{figure}
    \centering
    
    \begin{tikzpicture}
    \draw[black, thick] (-3,-2) -- (3,2);
    \filldraw[black] (-1,-2/3) circle (1.5pt);
    \filldraw[black] (-1.5,-1) circle (1.5pt);
    \filldraw[black] (1.2,0.8) circle (1.5pt);
    \filldraw[black] (2.7,1.8) circle (1.5pt);
    
    \draw[black, thick] (1.5,-2) -- (-1.5,2);
    \filldraw[black] (1,-4/3) circle (1.5pt);
    \filldraw[black] (0.6,-0.8) circle (1.5pt);
    \filldraw[black] (-0.3,0.4) circle (1.5pt);
    \filldraw[black] (-1.2,1.6) circle (1.5pt);
    
    \draw[black, thick] (-1,-2) -- (1,2);
    \filldraw[black] (-0.9,-1.8) circle (1.5pt);
    \filldraw[black] (-0.8,-1.6) circle (1.5pt);
    \filldraw[black] (-0.6,-1.2) circle (1.5pt);
    \filldraw[black] (0.2,0.4) circle (1.5pt);
    \filldraw[black] (0.45,0.9) circle (1.5pt);
    \filldraw[black] (0.9,1.8) circle (1.5pt);
    
    \draw[black, thick] (-3,-1) -- (3,1);
    \filldraw[black] (-2,-2/3) circle (1.5pt);
    \filldraw[black] (1.5,0.5) circle (1.5pt);

    \draw[red, thick] (0,-2) -- (0,2);
    \filldraw[red] (0,-1.3) circle (1.5pt);
    \filldraw[red] (0,0.5) circle (1.5pt);
    \filldraw[red] (0,1.2) circle (1.5pt);

    \draw[red, thick] (-1.5,-2) -- (1.5,2);
    \filldraw[red] (0.75,1) circle (1.5pt);
    \filldraw[red] (0.9,1.2) circle (1.5pt);
    \filldraw[red] (1.4,5.6/3) circle (1.5pt);
    \filldraw[red] (-0.2,-0.8/3) circle (1.5pt);
    \filldraw[red] (-0.5,-2/3) circle (1.5pt);
    \filldraw[red] (-1,-4/3) circle (1.5pt);
    \filldraw[red] (-1.2,-1.6) circle (1.5pt);

    \draw[red, thick] (-3,-0.5) -- (3,0.5);
    \filldraw[red] (2,1/3) circle (1.5pt);

    \draw[red, thick] (-3,0.75) -- (3,-0.75);
    \filldraw[red] (2.5,-0.625) circle (1.5pt);
    \filldraw[red] (1.5,-0.375) circle (1.5pt);
    \filldraw[red] (-2,0.5) circle (1.5pt);
    \filldraw[red] (-1,0.25) circle (1.5pt);
    \filldraw[red] (-0.8,0.2) circle (1.5pt);

    \draw[red, thick] (-2,2) -- (2,-2);
    \filldraw[red] (0.8,-0.8) circle (1.5pt);
    \filldraw[red] (1.9,-1.9) circle (1.5pt);
    \filldraw[red] (-1.4,1.4) circle (1.5pt);
    
    \filldraw[black] (0,0) circle (1.5pt) node[anchor=west] {$a$};
\end{tikzpicture}

    \caption{A winning set for Alice with $\abs{S}$ even and $a \in S$.}
    \label{fig:even_2d}
\end{figure}
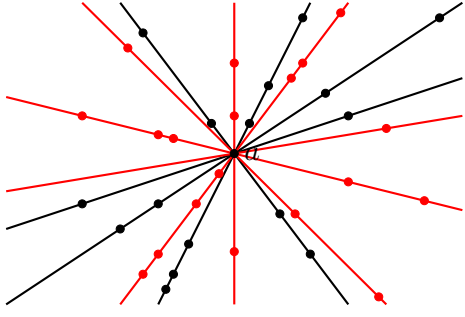

In fact, we will prove that, if the perfect lines are concurrent, then exactly one of the two conditions above must hold (depending on whether the common point of the perfect lines belongs to $S$ or not):

\begin{theorem}
\label{thm_even_2d}
When $d=2$ and $\abs{S}$ is even, Alice wins if and only if all the perfect lines are concurrent. In this case, the common point is a winning point for Alice.
\end{theorem}

\begin{proof}
    If the perfect lines are not concurrent, then Observation \ref{obs_even_2d} implies that Bob wins.
    
    Suppose now that the perfect lines are concurrent, say at $x$, and let $\abs{S}=2n$. Assume that neither condition from Proposition \ref{prop_even} holds. Let $L = \{l_1,\dots,l_t\}$ be the set of perfect lines labeled in a clockwise order, and define $l_i^+, l_i^-, u_i$, and $v_i$ as in the proof of Lemma \ref{lem_2d_even_perfect}. 
    
    We first consider the case $x \notin S$. Note that, if every set of $t$ consecutive (in a rotational sense) half-lines from $L$ contains a total of exactly $n$ points of $S$, then $u_i=v_i$ for all $i$, which means that the first condition from Proposition \ref{prop_even} holds, contradicting our assumption. We can thus assume, without loss of generality, that $u_1+\dots+u_t \geq n+1$. 
    
    Now consider the case $x \in S$. Suppose that there exists some $i$ for which $\abs{u_i-v_i} \geq 2$. Without loss of generality, suppose that $u_1 \geq v_1+2$, and also that $u_2+\dots+u_t \geq v_2+\dots+v_t$. Then $2(u_1+\dots+u_t) \geq (u_1+\dots+u_t) + (v_1+2)+(v_2+\dots+v_t) = 2n+1$, and thus $u_1+\dots+u_t \geq n+1$.
    
    Now suppose that $\abs{u_i-v_i} \leq 1$ for all $i$. Let $L^*=\{l_1',\dots,l_k'\}$ be the set of odd-balanced perfect lines, labeled in a clockwise order. Note that $k$ must be odd. Let $u_i'$ and $v_i'$ have the obvious meaning, and let $T' \coloneqq u_1'+\dots u_k'+v_1'+\dots+v_k'$. Since $u_i'+v_i'$ is odd for every $i$, and $k$ is odd, we deduce that $T'$ is odd. Since we are assuming that the second condition from Proposition \ref{prop_even} is not satisfied, there exist two consecutive lines in $L^*$ that both have the `extra' point of $S$ on the same side of $x$ (either both to the left of $x$, or both to the right of $x$). Without loss of generality, suppose that $u_1'=v_1'+1$, $u_2'=v_2'+1$, and also that $u_3'+\dots+u_k' \geq v_3'+\dots+v_k'$. Then $2(u_1'+\dots+u_k') \geq (u_1'+\dots+u_k')+(v_1'+v_2'+2)+(v_3'+\dots+v_k')=T'+2$. Since $T'$ is odd, we get $u_1'+\dots+u_k' \geq \frac{T'+3}{2}$. Since every line $l_i \notin L^*$ satisfies $u_i=v_i$, we can find $t$ consecutive half-lines from $L$ that contain at least $\frac{T'+3}{2} + \frac{2n-1-T'}{2} = n+1$ points of $S$ (excluding $x$), so again, without loss of generality, $u_1+\dots+u_t \geq n+1$.

    Therefore, we have, in any case, that $U \coloneqq u_1+\dots+u_t \geq n+1$. Since we also know that $u_2+\dots+u_t \leq n-1$, and $u_1+\dots+u_{t-1} \leq n-1$, we deduce that $u_1 \geq 2$, and $u_t \geq 2$. Let $p$ be the point of $S$ on $l_t^+$ that lies closest to $x$, and let $q$ be the point of $S$ on $l_1^+$ that lies farthest away from $x$, as shown in Figure \ref{fig:even_perfect}. Let $s_1,\dots,s_z$ be the lines that contain $p$ and at least one other point of $S$ that belongs to $l_i^+$ for some $i \in \{1,\dots,t-1\}$. The $s_i$ are labeled in a clockwise order. Let $w_i$ be the number of points of $S$ on $s_i$, not counting $p$.

     Let $s_{i_*}$ be the line through $p$ and $q$. We will show that $s_i$ is a perfect line for some $i \leq i_*$, contradicting the fact that every perfect line passes through $x$. We start with the inequality $2n-U \leq n-1$. Let $j \in \{1,\dots,i_*\}$ be smallest such that $2n-U+w_1+\dots+w_j \geq n$. Note that $w_1+\dots+w_{i_*} \geq u_1$ (since $i_* \geq u_1$), and thus $2n-U+w_1+\dots+w_{i_*} \geq 2n-(u_2+\dots+u_t) \geq n+1>n$, so such $j$ does exist. We show that the line $s_j$ is perfect. Indeed, the number of points of $S$ to the left of $s_j$ is $2n-U+w_1+\dots+w_{j-1} \leq n-1$, by minimality of $j$. Also, the number of points of $S$ to the right of $s_j$ is $2n - (2n-U+w_1+\dots+w_{j-1}) - (w_j+1) = U-(w_1+\dots+w_j)-1 \leq n-1$, by the choice of $j$.
\end{proof}

\begin{figure}
    \centering
    
    \begin{tikzpicture}
    \draw[black, thick] (0,0) -- (2,3) node[anchor=west] {$l_1^+$};
    \filldraw[black] (0.9,1.35) circle (1.5pt);
    \filldraw[black] (1.2,1.8) circle (1.5pt);
    \filldraw[black] (1.5,2.25) circle (1.5pt) node[anchor=west] {$q$};
    
    \draw[black, thick] (0,0) -- (3,2);
    \filldraw[black] (0.45,0.3) circle (1.5pt);
    \filldraw[black] (1.8,1.2) circle (1.5pt);
        
    \draw[black, thick] (0,0) -- (3,1);
    \filldraw[black] (0.99,0.33) circle (1.5pt);
    \filldraw[black] (2.4,0.8) circle (1.5pt);
    
    \draw[black, thick] (0,0) -- (2.25,-3);
    \filldraw[black] (1,-4/3) circle (1.5pt);
    \filldraw[black] (0.6,-0.8) circle (1.5pt);
    
    \draw[black, thick] (0,0) -- (3,-2.4);
    \filldraw[black] (2,-1.6) circle (1.5pt);

    \draw[black, thick] (0,0) -- (0,-3) node[anchor=west] {$l_t^+$};
    \filldraw[black] (0,-1.2) circle (1.5pt) node[anchor=east] {$p$};
    \filldraw[black] (0,-2) circle (1.5pt);

    \draw[black, thick] (-0.54,-3) -- (1.26,3) node[anchor=east] {$s_1$};

    \draw[black, thick] (-18/23,-3) -- (42/23,3) node[anchor=south] {$s_{i_*}$};
    
    \filldraw[black] (0,0) circle (1.5pt) node[anchor=east] {$x$};
\end{tikzpicture}

    \caption{Finding a perfect line that does not contain $x$.}
    \label{fig:even_perfect}
\end{figure}
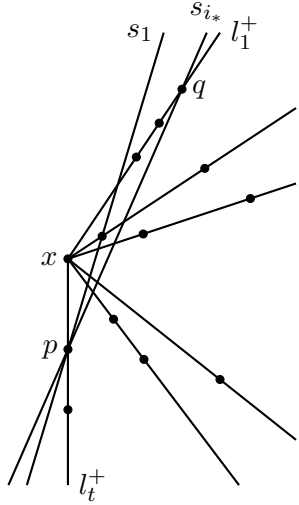

\section{The result in higher dimensions}

Given any choice of $a$, and any hyperplane $H$ through $a$, Bob can always claim all the points of $S$ in $H^+$, by choosing a point $b \in H^+$ sufficiently close to $a$ on the perpendicular to $H$ at $a$. We will prove a strengthening of this fact, which says that, if $K$ is an affine subspace contained in $H$, then Bob can additionally claim a lot of points on $K$. We will later apply this result with $K$ being either a line or a plane.

\begin{lemma}
\label{lem_at_least_half}
    Let $K \subset \mathbb{R}^d$ be an affine subspace of dimension $j \leq d-1$ containing $a$. Let $K' \subset K$ be an affine subspace of dimension $j-1$ containing $a$, and let $K^+ \subset K$ be one of the two open half-spaces of $K$ defined by $K'$. Let $H$ be a hyperplane that contains $K$. Then Bob can claim all the points of $S$ that belong to $H^+ \cup K^+$.
\end{lemma}

\begin{proof}
    Without loss of generality, let $a$ be the origin, $K=\textrm{span} \{e_1,\dots,e_j\}$, $K'=\textrm{span} \{e_2,\dots,e_j\}$, $K^+ = \{x \in K : x_1>0\}$, and $H=\{ x \in \mathbb{R}^d: \langle x,e_d \rangle = 0\}$. Let $H_1 = \{ x \in \mathbb{R}^d: \langle x,e_d + \epsilon e_1 \rangle = 0\}$, for some small $\epsilon>0$.

    If $x \in K^+ \cap S$, then $x_d=0$, so $\langle x,e_d + \epsilon e_1 \rangle = \epsilon x_1 >0$, and therefore $x \in H_1^+$. If $x \in H^+ \cap S$, then $x_d>0$, so $\langle x,e_d + \epsilon e_1 \rangle = x_d + \epsilon x_1 > 0$, for sufficiently small $\epsilon$, and therefore $x \in H_1^+$. Since we only need to deal with finitely many points on $H^+ \cap S$, some small enough $\epsilon$ will work for all of them. We have shown that $(H^+ \cup K^+) \cap S \subset H_1^+$. Since Bob can claim all the points of $S$ that belong to $H_1^+$, we are done.
\end{proof}

\subsection{Odd number of voters}

As in the two-dimensional case, we begin by stating that Alice has to play on every perfect hyperplane, and we also show that a lot of perfect hyperplanes exist:

\begin{obs}
\label{obs_odd_general}
If the hyperplane $H$ is perfect (or just good), and $a \notin H$, then Bob wins by choosing $b$ to be the projection of $a$ onto $H$.
\end{obs}

\begin{lemma}
\label{lem_odd_perfect}
Let $c \in \mathbb{R}^d$ be non-zero. A perfect hyperplane of the form $\{ x \in \mathbb{R}^d: \langle x,c \rangle = \lambda\}$ (for some $\lambda \in \mathbb{R}$) exists if, for any distinct $x_1, x_2 \in S$, $\langle x_1,c \rangle \neq \langle x_2,c \rangle$.
\end{lemma}

\begin{proof}
As previously observed, if $S = \{x_1,\dots,x_n\}$, then taking $\lambda$ to be the median of $\langle x_1,c \rangle,\dots,\langle x_n,c \rangle$ works.
\end{proof}

\begin{proof}[Proof of Theorem \ref{thm_odd_general}.]
Suppose there exists $x$ satisfying (\ref{condition1_odd}). Alice chooses $a=x$. Then, whatever the choice of $b$ is, we can think of $a$ as being the projection of $b$ onto some good hyperplane, showing that Alice wins.

If Alice wins, then Observation \ref{obs_odd_general} implies that (\ref{condition2_2d_odd}) necessarily holds.

Now suppose there exists a point $x$ satisfying (\ref{condition2_odd}).  We will show that $x$ also satisfies (\ref{condition1_odd}). Without loss of generality, suppose $x$ is the origin, and assume that the hyperplane given by $\langle x,e_1 \rangle=0$ is not good. By changing our coordinate system if necessary, we can further assume that no two points of $S$ agree in both the first and second coordinates. 

Let $\abs{S} = n$, and consider hyperplanes of the form $\langle x,e_1+\epsilon e_2 \rangle=0$. We fix distinct $x,y \in S$, and say that $\epsilon \neq 0$ is {\em bad} for $\{x,y\}$ if $\langle x-y,e_1+\epsilon e_2 \rangle = 0$. If $x$ and $y$ agree in the second coordinate, and $\epsilon$ is bad for $\{x,y\}$, then $x$ and $y$ also agree in the first coordinate, which is not possible. If $x$ and $y$ differ in the second coordinate, and both $\epsilon$ and $\epsilon'$ are bad for $\{x,y\}$, then $(\epsilon-\epsilon') \langle x-y,e_2 \rangle=0$, which implies that $\epsilon = \epsilon'$. Therefore, there is at most one bad value of $\epsilon$ for any pair of points of $S$, and thus at most finitely many values of $\epsilon$ for which $\langle x,e_1+\epsilon e_2 \rangle$ takes repeated values as $x$ varies over $S$.

 Therefore, using Lemma \ref{lem_odd_perfect}, we can find some $\epsilon>0$ such that both $H_1 = \{x \in \mathbb{R}^d:\langle x,e_1 + \epsilon e_2 \rangle=0\}$ and $H_2 = \{x \in \mathbb{R}^d:\langle x,e_1-\epsilon e_2 \rangle=0\}$ are perfect. We can further ensure that $\epsilon$ is sufficiently small, so that, for any $x \in S$, if $\langle x,e_1 \rangle>0$, then $\langle x,e_1+\epsilon e_2 \rangle>0$, and if $\langle x,e_1 \rangle<0$, then $\langle x,e_1+\epsilon e_2 \rangle<0$. Suppose that $k$ points of $S$ satisfy $\langle x,e_1 \rangle>0$, and $k'$ points of $S$ satisfy $\langle x,e_1 \rangle<0$, and also that, among those points of $S$ that satisfy $\langle x,e_1 \rangle=0$, $p$ satisfy $\langle x,e_2 \rangle>0$, and $p'$ satisfy $\langle x,e_2 \rangle<0$. Both $H_1$ and $H_2$ being perfect implies that $k+p=k'+p'$, and $k+p'=k'+p$, so $k=k'$ (and $p=p'$). But this contradicts $\langle x,e_1 \rangle=0$ not being good.

 If (\ref{condition3_odd}) holds, then Alice wins by picking $a=x$. Indeed, whatever the choice of $b$ is, for every line $l$ through $x$, Alice either claims all the points of $S$ on $l^+$, or she claims all the points of $S$ on $l^-$, and she also claims $x$. This guarantees a total of more than half of the votes.

 Finally, suppose no point of $S$ satisfies (\ref{condition3_odd}), and Alice picks some point $x \in S$. We find some line $l$ through $x$ that is not even-balanced. Suppose that $l^+$ contains more points of $S$ than $l^-$ does. We now find some hyperplane $H$ that contains $l$ such that $(H \setminus l) \cap S = \emptyset$. Suppose $H^+$ contains at least as many points of $S$ as $H^-$ does. Using Lemma \ref{lem_at_least_half}, with $K=l$, shows that Bob can claim all the points of $S$ in $l^+ \cup H^+$, which guarantees more than half of the votes for him.
\end{proof}

\subsection{Even number of voters}
 Once again, (the analogue of) Observation \ref{obs_odd_general} holds true:
 
\begin{obs}
\label{obs_even_general}
If the hyperplane $H$ is perfect, and $a \notin H$, then Bob wins by choosing $b$ to be the projection of $a$ onto $H$.
\end{obs}

It is possible to follow the idea of the proof of Lemma \ref{lem_2d_even_perfect} in order to prove the existence of perfect hyperplanes through any $x \in S$. However, we will in fact use the two-dimensional result in order to prove a stronger result, which states that we can further ask for the perfect hyperplane to contain any $(d-2)$-dimensional affine subspace containing $x$.

\begin{lemma}
\label{even_perfect_strong}
    Let $S$ be a set of $2n$ points in $\mathbb{R}^d$. Let $x$ be any point of $S$, and let $K$ be an affine subspace of dimension $d-2$ that contains $x$. Then there exists a perfect hyperplane $H$ that contains $K$.
\end{lemma}

\begin{proof}
    For simplicity, we take $x$ to be the origin, and $K = \textrm{span}\{e_3, \dots, e_d\}$. Then we need to find a perfect hyperplane $H$ of the form $\{x \in \mathbb{R}^d : \lambda x_1 - \mu x_2 = 0 \}$, for some $(\lambda,\mu) \in \mathbb{R}^2 \setminus \{(0,0)\}$. Let $\Pi:\mathbb{R}^d \to \mathbb{R}^2$ be given by $\Pi(x_1, x_2, \dots, x_d) = (x_1, x_2, 0, \dots, 0)$. In other words, $\Pi$ denotes the projection onto the two-dimensional plane $K^\bot = \textrm{span}\{e_1,e_2\}$. Let $S'$ be the image of $S$ under $\Pi$. Then we can apply Lemma \ref{lem_2d_even_perfect} to find a perfect line $l$ through the origin in the plane, noting that $S'$ possibly being a multiset does not affect our proof. Let $l$ be given by $\lambda x_1-\mu x_2 = 0$, for some $(\lambda,\mu) \in \mathbb{R}^2 \setminus \{(0,0)\}$. Then the hyperplane given by the same equation in $\mathbb{R}^d$ is also perfect.
\end{proof}

In order to show that a winning point for Alice should satisfy condition (\ref{condition2_even}), we will use (a corollary of) the Sylvester-Gallai theorem, which was posed as a problem by Sylvester in 1893 and proved by Gallai in 1944 \cite{gallai}, and states the following:

\begin{theorem}
    \label{thm_syl_gallai}
    Let $S$ be a finite set of points on a plane, such that any line that passes through two points of $S$ also contains a third point of $S$. Then all the points of $S$ lie on a single line.

\end{theorem}

\begin{cor}
\label{cor_Gallai}
    Let $L$ be a finite set of lines through the origin in $\mathbb{R}^d$, with $\abs{L} \geq 3$. Suppose that no plane contains exactly two elements of $L$. Then there exists a plane containing all the lines in $L$.
\end{cor}

\begin{proof}
    Since $L$ is finite, we can find some hyperplane $H$ that does not contain the origin, and intersects every line from $L$. Let $L = \{l_1, \dots, l_t\}$, let $x_i=l_i \cap H$, and let $S = \{x_1, \dots, x_t\}$. Let $l^*$ be the line through $x_i$ and $x_j$, for some $i \neq j$. Using the defining property of $L$, we deduce that the plane containing $l_i$ and $l_j$ must also contain some other line from $L$, say $l_k$. But then $x_k$ belongs to $l^*$, so any line containing two points of $S$ also contains a third one. Theorem \ref{thm_syl_gallai} implies that all the $x_i$ are collinear, which means that all the $l_i$ belong to a single plane.
\end{proof}

\begin{proof}[Proof of Theorem \ref{thm_even_general}.]
    If (\ref{condition1_even}) does not hold for any $x$, then Observation \ref{obs_even_general} implies that Bob wins.

    Now suppose that (\ref{condition1_even}) holds, and the perfect hyperplanes are concurrent at the origin. Let $\abs{S}=2n$. Alice chooses $a$ to be the origin. We would like to show that she wins. Using Theorem \ref{general_easy}, we are done if any hyperplane through the origin contains at most $n$ points of $S$ on either (strict) side. Assume this is not true. Without loss of generality, assume that $x_1>0$ holds for at least $n+1$ points of $S$. By changing our coordinate system if necessary, we can further assume that no two points of $S$ agree in both the first and second coordinates.  Using Lemma \ref{even_perfect_strong}, and the fact that every perfect hyperplane contains the origin, we deduce that any point of $S$ belongs to some perfect hyperplane of the form $\{x \in \mathbb{R}^d : \lambda x_1-\mu x_2=0\}$, for some $(\lambda,\mu) \in \mathbb{R}^2 \setminus \{(0,0)\}$. But now we can reduce the problem to its two-dimensional version, by projecting everything onto the plane $K = \textrm{span}\{e_1,e_2\}$, just like we did in the proof of Lemma \ref{even_perfect_strong}. Since $x_1>0$ holds for at least $n+1$ points of $S$, Theorem \ref{general_easy} implies that the origin is not a winning point for Alice in the two-dimensional game on $K$. Therefore, using Theorem \ref{thm_even_2d}, we deduce that there exists some perfect line on $K$ that does not contain the origin, and consequently there exists a perfect hyperplane in $\mathbb{R}^d$ that does not contain the origin. This gives a contradiction, so Alice wins.

    Suppose now that $x$ satisfies (\ref{condition2_even}), and Alice chooses $a=x$. Whatever the choice of $b$ is, Alice claims at least half the points of $S$ on any even-balanced line (excluding $x$). If $x \notin S$, then she wins. If $x \in S$, then there is an odd number of odd-balanced lines. Let $2N+1$ be the total number of points of $S$ on the odd-balanced lines, excluding $x$. By considering the projection of $b$ onto $P$, and looking at the proof of Proposition \ref{prop_even}, we deduce that Alice claims at least $N$ of those points, and therefore claims at least half of the points in total (as she also claims $x$), so she wins.

    Now suppose that no point satisfies (\ref{condition2_even}), and Alice chooses $a=x$. If there exists any non-balanced line $l$ through $x$, then we can assume that $\abs{l^+ \cap S} \geq \abs{l^- \cap S} + 2$. We can find some hyperplane $H$ that contains $l$, such that $(H \setminus l) \cap S = \emptyset$. Suppose that $H^+$ contains at least as many points of $S$ as $H^-$ does. Applying Lemma \ref{lem_at_least_half}, with $K=l$, shows that Bob can claim all the points of $S$ in $l^+ \cup H^+$, which guarantees more than half of the votes for him. 
    
    Otherwise, every line through $x$ is balanced. We can assume that there are at least two odd-balanced lines, as $x$ would trivially satisfy (\ref{condition2_even}) in a different case. If the odd-balanced lines are not coplanar, then Corollary \ref{cor_Gallai} implies that there exists some plane $Q$ that contains precisely two odd-balanced lines. Then we can find a line through $x$ on $Q$, containing no other points of $S$, that divides it into the half-planes $Q^+$ and $Q^-$, such that $\abs{Q^+ \cap S} = \abs{Q^- \cap S} + 2$. We can find some hyperplane $H$ containing $Q$, and no points of $S$ outside of $Q$. Suppose that $\abs{H^+ \cap S} \geq \abs{H^- \cap S}$. Applying Lemma \ref{lem_at_least_half}, with $K=Q$, shows that Bob can claim all the points of $S$ in $Q^+ \cup H^+$, which guarantees more than half of the votes for him.
    
    Finally, if the odd-balanced lines all belong to a single plane $P$, but (\ref{condition2_even}) is not satisfied, then we can assume (following the notation in (\ref{condition2_even})), without loss of generality, that $\abs{l_1^+ \cap S} = \abs{l_1^- \cap S} + 1$, and $\abs{l_2^+ \cap S} = \abs{l_2^- \cap S} + 1$. We can further assume that $l_3^+ \cup \dots \cup l_k^+$ contains at least as many points of $S$ as $l_3^- \cup \dots \cup l_k^-$ does. Then we can find a line through $x$ on $P$, containing no other points of $S$, that divides it into the half-planes $P^+$ and $P^-$, such that $P^+$ contains $l_i^+$ for all $i$. Then $\abs{P^+ \cap S} \geq \abs{P^- \cap S}+2$. We can find some hyperplane $H$ containing $P$ and no points of $S$ outside of $P$. Since every line through $x$ that does not belong to $P$ is even-balanced, we have $\abs{H^+ \cap S} = \abs{H^- \cap S}$. Applying Lemma \ref{lem_at_least_half}, with $K=P$, shows that Bob can claim all the points of $S$ in $P^+ \cup H^+$, which guarantees more than half of the votes for him.
\end{proof}

\subsection{How `big' is the set of Alice wins?}

Let $n,d$ be fixed positive integers. Consider the metric space
$$M = M_{n,d} \coloneqq \{ (x_1, \dots, x_n) \in (\mathbb{R}^d)^n: \text{all the} \ x_i \ \text{are distinct}\}$$
endowed with the usual Euclidean metric $\rho$; this is also a measure space, inheriting the Lebesgue measure from $(\mathbb{R}^d)^n$. Let
$$A = A_{n,d} \coloneqq \{(x_1, \dots, x_n) \in M_{n,d} : \{x_1, \dots, x_n\} \text{ is an Alice win}\}$$
denote the set of Alice wins\footnote{Or, to be pedantic, the set of tuples corresponding to Alice wins}, and let $B = B_{n,d} \coloneqq M_{n,d} \setminus A_{n,d}$ be the set of Bob wins. We are interested in determining how `big' $A$ is, as a subset of $M$. It is easy to see that, if $S$ is a Bob win, then there exists $\epsilon>0$ such that, if we replace each point $x$ of $S$ by another point $x'$ at distance at most $\epsilon$ from $x$ (where possibly $x'=x$), the resulting set is still a Bob win. Therefore, $B$ is an open subset of $M$.

Since Alice always wins when $d=1$, we have $A_{n,1} = M_{n,1}$ for all $n$. Alice also wins when $n \leq 2$, so $A_{1,d} = M_{1,d}$, and $A_{2,d} = M_{2,d}$, for all $d$. Moreover, when $d=2$ and $n=4$, Alice always wins. Indeed, to see this, let $S=\{x_1, \dots, x_4\}$ be a set of four points in the plane. If any three points are collinear, then there exist distinct $i_1,i_2,i_3$ such that $x_{i_3}$ belongs to the interior of the line segment $x_{i_1}x_{i_2}$. Then Alice wins by choosing $a=x_{i_3}$. If $x_{i_1},x_{i_2},x_{i_3}$ form a triangle, and $x_{i_4}$ lies in the interior of this triangle, then Alice wins by choosing $a=x_{i_4}$, as every perfect line contains this point, so condition (\ref{condition1_even}) of Theorem \ref{thm_even_general} is satisfied. Finally, if the four points form a convex quadrilateral, $x_{i_1}x_{i_2}x_{i_3}x_{i_4}$ say, then Alice wins by choosing $a$ to be the point of intersection of the lines $x_{i_1}x_{i_3}$ and $x_{i_2}x_{i_4}$, since every perfect line contains this point. Therefore, $A_{4,2} = M_{4,2}$.

Now let $d=2$ and $n\geq6$ even. Let $x_1, \dots, x_{n-1}$ be the vertices of a regular polygon, and let $x_n$ be the center of this polygon. This is a winning set for Alice, because $x_n$ satisfies condition (\ref{condition2_even}) of Theorem \ref{thm_even_general}. It is clear that there exists $\epsilon>0$ such that if $\rho(x_i,x_i') < \epsilon$ for each $i$, then replacing $x_i$ with $x_i'$ (for each $i$) still results in an Alice win. Therefore, $A$ contains a non-empty open set. The set of Bob wins is also non-empty in this case: for example, let $x_1, \dots, x_n$ lie on a circle $C$ in a clockwise order, such that they form the vertices of a regular polygon. We replace $x_1$ by some $x_1'$, such that $x_1'$ is very close to $x_1$, and still belongs to $C$. For $i=2, \dots, \frac{n}{2}$, the perfect line $x_ix_{\frac{n}{2}+i}$ passes through the center $O$ of $C$. However, $x_1'x_{\frac{n}{2}+1}$ does not pass through $O$, which implies that Bob wins. We observed above that $B$ is always an open set. Therefore, when $n \geq 6$ is even, both $A_{n,2}$ and $B_{n,2}$ contain a non-empty open set, and therefore have positive measure (in fact infinite measure, as they are translation-invariant).

A finite set $S \subset \mathbb{R}^d$ is in \emph{general position} if, for $i=1, \dots, d-1$, any affine subspace of dimension $i$ contains at most $i+1$ points of $S$. We will now show that, for any other pair $(n,d)$, the set $A_{n,d}$ contains no sets of points in general position, which implies that it is nowhere dense in $M_{n,d}$.

\begin{lemma}
\label{lem_general_position}
     Let $\abs{S} \geq 3$. If $d=2$ and $\abs{S}$ is odd, or $d \geq 3$, then any $S$ in general position is a Bob win.
\end{lemma}

\begin{proof}
    Let $d \geq 2$, and $\abs{S}$ odd. Suppose that Alice can win by choosing $a=x$, for some $x \in S$. Then, using condition (\ref{condition3_odd}) from Theorem \ref{thm_odd_general}, along with the fact that $S$ contains at least three points, we deduce that there exists some line that contains $x$ and at least two other points of $S$, which implies that $S$ is not in general position.

    Now let $d \geq 3$, and $\abs{S}$ even. Suppose that Alice can win by choosing $a=x$, for some $x \notin S$. Condition (\ref{condition2_even}) from Theorem \ref{thm_even_general} implies that every line through $x$ is even-balanced. If all the points of $S$ belong to a single line, then $S$ is not in general position, as it has size at least four. If $l_1$ and $l_2$ are two distinct even-balanced lines, each containing at least two points of $S$, then the plane containing $l_1$ and $l_2$ contains at least four points of $S$, which implies that $S$ is not in general position. Finally, suppose that Alice can win by choosing $a=x$, for some $x \in S$. If there is any non-trivial even-balanced line through $x$, then that line contains at least three points of $S$, and therefore $S$ is not in general position. If every line that contains $x$ and at least one other point of $S$ is odd-balanced, then condition (\ref{condition2_even}) of Theorem \ref{thm_even_general} implies that all the points of $S$ belong to a single plane. Since $\abs{S} \geq 4$, $S$ is not in general position.
\end{proof}

\begin{cor}
    Let $n \geq 3$. If $d=2$ and $n$ is odd, or $d \geq 3$, then $A_{n,d}$ is nowhere dense in $M_{n,d}$.
\end{cor}

\begin{proof}
    Let $G = G_{n,d} \coloneqq \{ (x_1, \dots, x_n) \in M_{n,d}: \{x_1, \dots, x_n\} \text{ is in general position}\}$. Using Lemma \ref{lem_general_position}, we deduce that $G \subset B$. Since $G$ is an open dense subset of $M$, the complement of $B$, which is $A$, is nowhere dense in $M$.
\end{proof}

\section{Algorithms for determining the winning point for Alice}

Let $f$ and $g$ be functions taking values in the non-negative real numbers. We write $f(x) = O(g(x))$ if there exists some positive constant $C$ such that $f(x) \leq Cg(x)$ for all $x$.

We will now provide polynomial time algorithms that will allow us to check whether there exists a winning point for Alice, and identify the unique winning point when it exists. Below we will illustrate how this can be achieved using the condition on concurrent lines, as this condition seems to give the most efficient algorithm. However, it is also possible to use the conditions on good/perfect hyperplanes. As an example, we briefly explain how condition (\ref{condition1_odd}) of Theorem \ref{thm_odd_general} can be checked in finite time, for some $x \in S$. Without loss of generality, let $x$ be the origin, and let $x_1,\dots,x_{n-1}$ be all the points of $S$ other than the origin. For $i=1,\dots,n-1$, let $H_i = \{x \in \mathbb{R}^d: \langle x,x_i \rangle=0\}$. This gives a collection of (at most) $n-1$ hyperplanes through the origin, which divide $\mathbb{R}^d$ into a finite collection of open regions, say $R$. We observe that, if two hyperplanes $H$ and $H'$ through the origin have their normal vectors in the same region from $R$, then $H$ is good if and only if $H'$ is good. Therefore, we only need to check condition (\ref{condition1_odd}) for $\abs{R}$ hyperplanes. Using induction on both $d$ and the size of $S$, we can show that $\abs{R} = O(\abs{S}^d)$. This is a polynomial bound for fixed $d$, but we will now show how to get a polynomial bound that is independent of $d$.

\subsection{Odd number of voters}

When $\abs{S}$ is odd, we will check condition (\ref{condition3_odd}) of Theorem \ref{thm_odd_general} for every point of $S$. To check whether a point $x \in S$ satisfies this condition, we simply need to identify the set $L$ of lines that contain $x$ and at least one other point of $S$, and check if every line in $L$ contains an equal number of points of $S$ on either side of $x$. Since there are $\abs{S}$ points to check, and at most $\abs{S}$ lines to check for each point (in fact, there are at most $\abs{S}/2$ lines to check, as we immediately stop if we find a line containing only one additional point of $S$), this algorithm runs in time $O(\abs{S}^2)$.

\subsection{Even number of voters}

When $\abs{S}$ is even, we will check condition (\ref{condition2_even}) of Theorem \ref{thm_even_general}. Let $L$ be the set of lines that contain at least two points of $S$. We can assume that $\abs{L}>1$, since $\abs{L}=1$ corresponds to the one-dimensional case. Using condition (\ref{condition2_even}), we see that, if there exists a winning point $x$ for Alice, then it has to be either a point of $S$, in which case there exist some odd-balanced lines about $x$, or a point of intersection of two distinct lines from $L$. In the latter case, the two lines are non-trivial even-balanced lines about $x$. Since $\abs{L} \leq \binom{\abs{S}}{2}$, the number of potential winning points for Alice is $O(\abs{S}^4)$.

When dealing with a point $x \notin S$, we just repeat what we did in the odd case: we check that, for any line $l$ containing $x$ and at least one point of $S$, we have $\abs{l^+ \cap S} = \abs{l^- \cap S}$. As noted above, we need to check at most $\abs{S}/2$ lines.

When dealing with a point $x \in S$, we first check that, for any line $l$ containing $x$ and at least one other point of $S$, we have $\abs{\abs{l^+ \cap S} - \abs{l^- \cap S}} \leq 1$. Every time we come across an odd-balanced line, we also need to ensure that these lines still belong to a single plane. Finally, we check that the odd-balanced lines alternate in the way described in condition (\ref{condition2_even}). This can be done in time $O(\abs{S})$, and the entire process runs in time $O(\abs{S}^5)$.

\begin{remark}
\label{rem_rado}

We have shown that, provided $d > 2$ (or $d=2$ and $|S|$ is odd), Bob wins in `most' cases, in the sense that the set of Alice wins is nowhere dense. It is natural to ask for the maximal proportion $\alpha = \alpha(d) \in [0,1]$ such that, in the game on $\mathbb{R}^d$, Alice can always guarantee (with a suitable strategy) that she wins at least a proportion $\alpha$ of the votes. (Here, we insist on the rule that $b \neq a$, as if $b=a$ then neither player would obtain any votes, so we must rule out this `trivial' case in order for the question to be interesting.) It turns out that this is not too hard to answer, and in fact $\alpha(d) = 1/(d+1)$ for all $d \in \mathbb{N}$. We proceed to outline a proof of this, due to the author, David Ellis and Robert Johnson. We first provide an example where Alice cannot get more than $\frac{1}{d+1}$ of the votes. This construction also appears in \cite{regular_simplex} (proof of Theorem 4). We take $\abs{S}$ to consist of the $d+1$ vertices of a regular $d$-dimensional simplex in $\mathbb{R}^d$. In this case, whatever the choice of $a$ is, Bob can claim at least $d$ points by choosing his point to be the projection of $a$ onto some face of the regular simplex. Therefore, Alice cannot claim more than $\frac{\abs{S}}{d+1}$ points in general. (It is easy to see that this can also happen for arbitrarily large sets $S$, by placing $\lambda$ points in a small ball centered at each vertex of the regular $d$-dimensional simplex, for any $\lambda \in \mathbb{N}$.) We will now show that Alice can always claim at least $\frac{\abs{S}}{d+1}$ points, by applying the following result of Rado \cite{rado}.

\begin{theorem}
\label{thm_rado}
    Let $\mathcal{F}$ be a $\sigma$-algebra on $\mathbb{R}^d$, such that, for all $c \in \mathbb{R}^d$ and $t \in \mathbb{R}$, $\{ x \in \mathbb{R}^d : \langle x,c \rangle \leq t \} \in \mathcal{F}$. Let $(\mathbb{R}^d, \mathcal{F}, \mu)$ be a measure space, such that, for any $\epsilon > 0$, there exists some bounded set $T$ satisfying $\mu(\mathbb{R}^d \setminus T) < \epsilon$. Then there exists $a \in \mathbb{R}^d$, such that, for any $c \in \mathbb{R}^d$ and $t < \langle a,c \rangle$, $\mu(\{x \in \mathbb{R}^d: \langle x,c \rangle>t \}) \geq \frac{\mu(\mathbb{R}^d)}{d+1}$.
\end{theorem}

\begin{theorem}
    Given any finite set $S$ in $\mathbb{R}^d$, Alice can always claim at least $\frac{\abs{S}}{d+1}$ points.
\end{theorem}

\begin{proof}
    Let $\mu$ be the measure on all subsets of $\mathbb{R}^d$ defined as follows: for any $T \subseteq \mathbb{R}^d$, $\mu(T) = |T \cap S|$. Note that $S$ is bounded, and $\mu (\mathbb{R}^d \setminus S) = 0$, so we can apply Theorem \ref{thm_rado}, which implies that Alice can pick a point $a$ such that any closed half-space defined by some hyperplane through $a$ has measure at least $\frac{\mu({\mathbb{R}^d})}{d+1}$, or equivalently contains at least $\frac{\abs{S}}{d+1}$ points of $S$. Whatever the choice of $b$ is, we can think of $a$ as the projection of $b$ onto some hyperplane $H$ through $a$. Then the closed half-space defined by $H$ on the opposite side of $b$ contains at least $\frac{\abs{S}}{d+1}$ points of $S$, and Alice claims all of them.
\end{proof}

\end{remark}

\section*{Acknowledgments}
The author would like to thank both David Ellis and Robert Johnson for suggesting some of the problems considered here, for useful discussions and comments, and for allowing us to include Remark \ref{rem_rado}.

\bibliographystyle{plain}
\bibliography{references.bib}

\end{sloppypar}

\end{document}